\documentclass[letterpaper,11 pt]{article}
\usepackage{amsmath, amssymb, bbm, xspace}

\def\spose#1{\hbox to 0pt{#1\hss}}

\def\text #1{\hbox{\quad#1\quad}}

\def\nthinsp{\mskip -2   mu}

\def\superstar{^{\raise 0.5pt\hbox{$\nthinsp *$}}}
\def\SUPERSTAR{^{\raise 0.5pt\hbox{$*$}}}
\def\lamstarT {\lambda^{\raise 0.5pt\hbox{$\nthinsp *$}T}}

\def\hbar{\skew{4.2}\bar h}

		\def\bkE{{\rm I\kern-.17em E}}
		\def\bk1{{\rm 1\kern-.17em l}}
		\def\bkD{{\rm I\kern-.17em D}}
		\def\bkR{{\rm I\kern-.17em R}}
		\def\bkP{{\rm I\kern-.17em P}}
		\def\bkY{{\bf \kern-.17em Y}}
		\def\bkZ{{\bf \kern-.17em Z}}

		\def\beq{\begin{eqnarray}}
		\def\bc{\begin{center}}
		\def\be{\begin{enumerate}}
		\def\bi{\begin{itemize}}
		\def\bs{\begin{small}}
		\def\bS{\begin{slide}}
		\def\ec{\end{center}}
		\def\ee{\end{enumerate}}
		\def\ei{\end{itemize}}
		\def\es{\end{small}}
		\def\eS{\end{slide}}
		\def\eeq{\end{eqnarray}}

	\def\cp2problem#1#2#3#4{\fbox
		 {\begin{tabular*}{0.9\textwidth}
			{@{}l@{\extracolsep{\fill}}l@{\extracolsep{6pt}}l@{\extracolsep{\fill}}c@{}}
				#1 & & $#4 $ 
			\end{tabular*}}}
\def\z{\phantom 0}

		\renewcommand{\emph}[1]{\textbf{#1}}

		\def\bkE{{\rm I\kern-.17em E}}
		\def\bk1{{\rm 1\kern-.17em l}}
		\def\bkD{{\rm I\kern-.17em D}}
		\def\bkR{\mathbb{R}}
		\def\bkP{{\rm I\kern-.17em P}}
		
		\def\bkZ{{\bf{Z}}}

\newcommand {\beeq}[1]{\begin{equation}\label{#1}}
\newcommand {\eeeq}{\end{equation}}
\newcommand {\bea}{\begin{eqnarray}}
\newcommand {\eea}{\end{eqnarray}}

\def\texitem#1{\par\smallskip\noindent\hangindent 25pt
               \hbox to 25pt {\hss #1 ~}\ignorespaces}

\usepackage{graphicx}
\usepackage{amsfonts,amsmath,fullpage,bbm}
\usepackage{amssymb,amsthm,multirow,verbatim}
\usepackage{acronym,wrapfig,plain,mathrsfs,enumerate,relsize,color}
\newtheorem{algorithm}{Algorithm}
\usepackage{algorithm}
\usepackage{algorithm,algorithmic}
\newtheorem{theorem}{Theorem}
\newtheorem{remark}{Remark}

\newtheorem{lemma}{Lemma}

\newtheorem{proposition}{Proposition}
\newtheorem{assumption}{Assumption}

\usepackage{subfig}
\usepackage{algorithmic}
\usepackage{pifont}
\usepackage{footmisc}
\usepackage{stackengine}

\usepackage[pdftex,colorlinks=true,urlcolor=blue,citecolor=black,anchorcolor=black,linkcolor=black]{hyperref}
\newcommand{\qedd}{\tag*{$\square$}}
\usepackage{algorithmic}

\begin{document}
\allowdisplaybreaks
\title{A Projection-Based Algorithm  for Solving Stochastic Inverse Variational Inequality Problems}

% AUTHOR: Enter the authors of the article, see end of the example document for further examples
\author{Zeinab Alizadeh \footnote{ Systems and Industrial Engineering, University of Arizona,
 Tucson, AZ 85721, USA.} \and Felipe Parra Polanco \footnotemark[1] \and Afrooz Jalilzadeh \footnotemark[1]
}

\date{}

\maketitle

\section*{ABSTRACT}
We consider a stochastic Inverse Variational Inequality (IVI) problem defined by a  continuous and co-coercive map over a closed and convex set. Motivated by the absence of performance guarantees for stochastic IVI, we present a variance-reduced projection-based gradient method. Our proposed method ensures an almost sure convergence of the generated iterates to the solution, and we establish a convergence rate guarantee. To verify our results, we apply the proposed algorithm to a network equilibrium control problem. 

\section{INTRODUCTION}
\label{sec:intro}
In recent years, the inverse variational inequality (IVI) problem \cite{he1999goldstein} has become an active research area within the field of mathematical optimization. The IVI problem is pertinent to numerous fields, including  transportation system operation, control
policies, and electrical power network management \cite{yang1997traffic,he2011inverse,scrimali2012inverse,he2018existence}. Depending on the specific application, factors such as the demand market, supply market, or transaction cost may exhibit stochastic behavior.
In this paper, we focus on Stochastic Inverse Variational Inequality (SIVI). In particular, let $X\subseteq \mathbb R^n$ be a nonempty closed convex set, and $F:X\to \mathbb R^n$ be a continuous nonlinear map. Consider the following SIVI problem: find $x^* \in \mathbb R^n$ such that $F(x^*) \in X$ and
\begin{align}\label{sIvI}
\langle y-F(x^*), x^* \rangle\geq 0, \quad \forall y\in X,
\end{align}
where $F(x)\triangleq \mathbb E[G(x,\xi)]$, $\xi: \Omega \to \mathbb R^d$, ${G}: X \times \mathbb R^d  \rightarrow
\mathbb{R}^n$, and the associated probability space is denoted by $(\Omega, {\cal F}, \mathbb{P})$.
The SIVI problem involves finding a solution that satisfies a set of inequalities for all variable $y$ lying within the set $X$. Indeed, SIVI in \eqref{sIvI} can be viewed as a classical Stochastic Variational Inequality (SVI) problem, i.e., $\langle y-x^*,H(x^*)\rangle\geq 0$ for any $y\in X$, where $H\triangleq F^{-1}$. However, $F^{-1}$ may not be available in some practical applications which necessitates developing a new set of schemes. 

%Moreover, there is a close connection between SIVI and certain classes of optimization problems. Consider the minimization problem $\min_{z\in X} f(z)$. From the optimality condition of the minimization problem we have that $\langle\nabla f(z^*),y-z^*\rangle\geq 0$ for any $y\in X$. Let $x^*=\nabla f(z^*)$ and $F(x)=(\nabla f)^{-1}(x)$, then the optimality condition of the minimization problem is equivalent to solving SIVI problem \eqref{sIvI}.% which involves finding $z^*$ such that $F(z^*)\in X$ and $\langle z^*,y-F(z^*)\rangle$. 

While VIs \cite{YuriiNesterov2011DiscreteandContinuousDynamicalSystems,facchinei2007finite,malitsky2015projected} and SVIs \cite{jiang2008stochastic,koshal2012regularized,yousefian2017smoothing,jalilzadeh2019proximal} have been studied extensively over the last several decades, less is known about SIVIs. Projection-based algorithms to solve deterministic IVI problems have been developed by \cite{he2011inverse,he2018existence} and \cite{luo2014regularization}. Moreover, \cite{zou2016novel} introduced a neural network-based method to approximate the solutions to IVIs, however to the best of our knowledge there is no available algorithm with a convergence rate guarantee for solving the SIVI. In this paper, we introduce a variance-reduced projection-based algorithm for solving problem \eqref{sIvI} when operator $F$ is co-coercive, which is a weaker assumption than strong monotonicity plus Lipschitz continuity. We show that the proposed scheme produces a sequence that converges to the solution of \eqref{sIvI} almost surely, a convergence statement that was unavailable thus far for this class of problems. Then, we derive rate and complexity guarantees for the proposed algorithm.  

The remainder of the paper is organized as follows. In Section \ref{sec:pre}, we define our notations and state assumptions and some  technical lemmas that we use throughout the paper. In Section \ref{sec:method}, we propose the Variance-reduced Inverse Projected Gradient (VR-IPG) method to solve the SIVI problem \eqref{sIvI}, and its convergence analysis is stated in Section \ref{sec:rate}. Finally,  numerical experiments are presented in Section \ref{sec:numeric} to solve a network equilibrium control problem.

\section{PRELIMINARIES}\label{sec:pre}
In this section, first, we define important notations and then the main assumptions that we need for the convergence analysis are stated.

{\bf Notations.}  Throughout the paper, $\|x\|$ denotes the Euclidean vector norm, i.e., $\|x\|=\sqrt{x^Tx}$.  $\mathbf{P}_X [u]$ is the projection of $u$ onto the set $X$, i.e. $\mathbf{P}_X [u] = \mbox{argmin}_{z \in X} \| z-u \|$. $\mathbb E[x]$ is used to denote the expectation of a random variable $x$. Moreover, $e_n\in \mathbb R^n$ denotes the vector of ones and for a given $a\in \mathbb R^n$, $\mbox{diag}(a)\in \mathbb R^{n\times n}$ denotes a diagonal matrix whose main diagonal elements are $a$. 
\begin{assumption}\label{asump1}
Operator $F:X\rightarrow \mathbb R^n$ is a co-coercive map on $X$, i.e., there exists $\mu>0$ such that
\begin{align*}%\label{sm}
\langle F(x)-F(y),x-y\rangle\geq \mu \|F(x)-F(y)\|^2, \quad \forall x,y\in X.
\end{align*} 
%and \z{$L$-Lipschitz} continuous on $X$, i.e., there exists $L>0$ such that 
%\begin{align*}%\label{lip}
%\|F(x)-F(y)\|\leq L\|x-y\|,\quad \forall x, y \in X.
%\end{align*}

\end{assumption}
\begin{remark}
     It is worth noting that for a Lipschitz continuous map $F$, co-coercivity is a weaker assumption than strong monotonicity. In other words, a co-coercive map may not be strongly monotone, such as a constant mapping. However, a map that is both strongly monotone and Lipschitz continuous is co-coercive.
Moreover, co-coercive maps have some desirable properties. For example, co-coercivity is preserved under affine transformations. For instance, if $F$ is a co-coercive map, then $A^TF(Ax) + d$ is also co-coercive, where $A$ is a matrix and $d$ is a vector. This property highlights the versatility of co-coercivity in various settings. See \cite{zhu1996co} for more discussion about the properties of co-coercive maps.
\end{remark}

If $\mathcal F_{k}$ denotes the information history at epoch $k$, then we have the following requirements on the associated filtrations where $\bar w_{k,N_k} \triangleq \tfrac{1}{N_k}{\sum_{j=1}^{N_k} ( G(x_k,\xi_{j,k})-F(x_k))}$.
\begin{assumption}\label{assump_error}
 There exists $\nu>0$ such that $\mathbb E[\bar w_{k,N_{k}}\mid  \mathcal F_{k}]=0$ and $\mathbb E[\| \bar w_{k,N_{k}}\|^2\mid  \mathcal F_{k}]\leq \tfrac{\nu^2}{N_{k}}$  holds almost surely 
 for all $k$, 
where $\mathcal{F}_k \triangleq \sigma\{x_0, x_1, \hdots, x_{k-1}\}$.  
\end{assumption} 
In our analysis, the following technical lemma for projection mappings is used.
\begin{lemma}\label{proj}\cite{bertsekas2003convex}
\noindent Let $ X\subseteq \mathbb{R}^n $  be a nonempty closed and convex set. Then the followings hold:
(a) $\|\mathbf{P}_X [u]- \mathbf{P}_X [v]\| \leq \|u-v\| $ for all $ u,v \in \mathbb{R}^n$;
(b) $ (u-\mathbf{P}_X [u])^T(x-\mathbf{P} _X [u]) \leq 0 $ for all $u \in \mathbb{R}^n$ and $x \in X$.  
\end{lemma}
Moreover, to prove almost sure convergence of the iterates we use the following Robbins-Siegmund lemma (see Lemma 11 in  \cite{polyak1987introduction}). 
\begin{lemma}\label{a.s.}[Robbins-Siegmund]
    Let $v_k, u_k, \alpha_k, \beta_k$ be  nonnegative random variables and let
    $\mathbb E[v_{k+1}\mid \mathcal F_k]\leq (1+\alpha_k)v_k-u_k+\beta_k$, $\sum_{k=0}^\infty \alpha_k<\infty$ and $\sum_{k=0}^\infty u_k<\infty$ almost surely, where $\mathbb E[v_{k+1}\mid \mathcal F_k]$ denotes the conditional mathematical expectation for the given $v_0,\hdots,v_k,u_0,\hdots,u_k,\alpha_0,\hdots,\alpha_k,\beta_0,\hdots,\beta_k$. Then, $v_k\to v$ almost surely and $\sum_{k=0}^\infty u_k<\infty$ almost surely, where $v\geq 0$ is some random variable.
\end{lemma}
\section{PROPOSED METHOD}\label{sec:method}
In this section, we propose our algorithm for solving problem \eqref{sIvI}. First, we show that solving SIVI problem \eqref{sIvI} is equivalent to finding $x^*\in \mathbb{R}^n$ such that
\begin{equation}\label{opt-cond}
    F(x^*)=\mathbf{P}_X(F(x^*)-\eta x^*).
\end{equation} 
\begin{proposition}
    $x^*$ is a solution of problem \eqref{sIvI} if and only if $x^*$ is a solution of equation \eqref{opt-cond}.   
\end{proposition}
\begin{proof}
    We start by rewriting the projection equation \eqref{opt-cond} which is equivalent to 
    $F(x^*)\in \mbox{argmin}_{y\in X}\|y-(F(x^*)-\eta x^*)\|^2$. We observe that $F(x^*)\in X$ and since the objective function is convex, the first-order optimality condition is equivalent to finding a global solution. Hence, $ \langle y-F(x^*),F(x^*)-(F(x^*)-\eta x^*) \rangle \geq 0$ for all $y\in X$. Therefore, $\langle {y-F(x^*), x^*} \rangle \geq 0$ for all $y\in X$ and $F(x^*)\in X$ which means that $x^*$ is the solution of problem \eqref{sIvI}.
\end{proof}
To address the stochastic nature of the problem, we utilize a stochastic approximation (SA) scheme, which has shown effectiveness in various stochastic computational problems. Conventional SA schemes rely on a single or fixed batch of samples to estimate the expectation. Our proposed approach is a projected gradient-based scheme where, in each iteration, $F(x_k)$ is substituted by a gradually increasing sample average using a batch size of $N_k$, i.e., $\frac{1}{N_k}\sum_{j=1}^{N_k}G(x_k,\xi_{j,k})$. A projection step generates an intermediate iteration $z_k$ and the new iterate $x_{k+1}$ is then generated by taking a direction along $-(\frac{1}{N_k}\sum_{j=1}^{N_k}G(x_k,\xi_{j,k})-z_k)$ with step-size $1/\eta_k$ from the current iterate point $x_k$. The outline of the proposed variance-reduced inverse projected gradient (VR-IPG) method is displayed in Algorithm \ref{alg1}.
Moreover, to measure how far the iterates are from the optimal solution, we can examine the degree to which the optimality condition in \eqref{opt-cond} is violated. Therefore, we define the gap function $H:X\times \mathbb R_+\to \mathbb R_+$ such that  $H(x,\eta)\triangleq \tfrac{1}{\eta}\left(F(x)-\mathbf{P}_X\left(F(x)-\eta x\right)\right)$ which will be used to analyze the convergence rate of the proposed method.

\begin{algorithm}[htbp]
\caption{Variance-reduced Inverse Projected Gradient (VR-IPG) method}
\label{alg1}
{\bf Input}: $x_0\in X$, $\{\eta_{k},N_k\}_{k\geq 0}\subseteq \mathbb{R}_+$; \\
{\bf for $k=0,\hdots T-1$ do}\\
\mbox{ }
$z_k=\mathbf{P}_{X}\left[\frac{\sum_{j=1}^{N_k}G(x_k,\xi_{j,k})}{N_k}-\eta_{k} x_k\right]$;\\
\mbox{ } $x_{k+1}= x_k-\frac{1}{\eta_{k}}\left(\frac{\sum_{j=1}^{N_k}G(x_k,\xi_{j,k})}{N_k}-z_k\right)$; \\
{\bf end for}
\end{algorithm}

\subsection{Convergence Analysis}\label{sec:rate}

To obtain the convergence results of Algorithm \ref{alg1}, we state a one-step analysis in Lemma \ref{lem:bound}. To this end, we first establish the following technical result that will be utilized in the proof of Lemma \ref{lem:bound}. Related proofs are provided in the appendix.
\begin{lemma}\label{first bound}
Define $H(x_k,\eta_k)\triangleq \tfrac{1}{\eta_k}\left(F(x_k)-\mathbf{P}_X\left(F(x_k)-\eta_kx_k\right)\right)$ and let $x^* \in X$ be a solution point for problem \eqref{sIvI}. For any $x \in \mathbb R ^n$, the following holds
\begin{align*}%\label{bound_term_a}
    (x-x^*)^T H(x, \eta) \geq \big( 1-\tfrac{1}{4\mu \eta}\big) \| H(x, \eta) \|^2.
\end{align*}
\end{lemma}

\begin{lemma}\label{lem:bound}
Consider iterates $\{x_k\}_{k\geq 0}$ generated by Algorithm \ref{alg1} and suppose Assumptions \ref{asump1} and \ref{assump_error} hold, then the following for any $k\geq 0$:
\begin{align}\label{lemma bound}
    \mathbb E[\|x_{k+1}-x^*\|^2\mid \mathcal F_k]\leq \|x_k-x^*\|^2-\left(1-\tfrac{1}{2\eta_k\mu}\right)\|H(x_k,\eta_k)\|^2+\tfrac{\z{6}\nu^2}{\eta_k^2N_k}+\tfrac{2\nu}{\eta_k\sqrt{N_k}}\|x^*\|.
\end{align}
\end{lemma}

Now we are ready to present the main results of this paper. 
In Theorem \ref{th:a.s.}, we show that the iterates generated by Algorithm \ref{alg1} converge to a solution of problem \eqref{sIvI} almost surely. 
\begin{theorem}[Almost sure convergence]\label{th:a.s.}
Consider iterates $\{x_k\}_{k\geq 0}$ generated by Algorithm \ref{alg1} and suppose Assumptions \ref{asump1} and \ref{assump_error} hold. Let $\eta_k=\eta$ for some $\eta>0$ and suppose $\{N_k\}_k$ is an increasing sequence, such that $\sum_{k=0}^\infty \tfrac{1}{\sqrt{N_K}}<\infty$. Then $x_k\to x^*$ almost surely. 
\end{theorem}
\begin{proof}
    Consider \eqref{lemma bound}, %lemma \eqref{lem:bound}, 
    if $\eta_k=\eta$ for some $\eta>0$ and $\sum_{k=0}^\infty \tfrac{1}{\sqrt{N_K}}<\infty$, the requirements of Lemma \ref{a.s.} are satisfied and we can conclude that: 
    
    (i) $\|x_k-x^*\|^2$ is a convergent sequence. Hence, $\{x_k\}_k$ is a bounded sequence which implies that it has a convergent subsequence $x_{k_q}\to x_\#$ as $q\to \infty$ for some $x_\#\in X$;
    
    (ii) $\sum_{k=0}^\infty \left(1-\tfrac{1}{2\eta_k\mu}\right)\|H(x_k,\eta)\|^2<\infty$ almost surely. Since $\left(1-\tfrac{1}{2\eta_k\mu}\right)>0$, hence we have that $\|H(x_k,\eta)\|^2\to 0$ almost surely. From the definition of $H(x_k,\eta)$, we conclude that $\lim_{k\to \infty}F(x_k)-\mathbf P_X(F(x_k)-\eta x_k)\to 0$ which implies that $\lim_{q\to \infty}F(x_{k_q})-\mathbf P_X(F(x_{k_q})-\eta x_{k_q})\to 0$. 
    
    From (i), we know that $x_{k_q}\to x_\#$ almost surely, so from (ii), we have that $F(x_\#)-\mathbf P_X(F(x_\#)-\eta x_\#)= 0$ which means that $x_\#$ is a solution of \eqref{sIvI}. Now we show that $z_\#$ is the unique limit point. Note that \eqref{lemma bound} is true for any optimal solution, so one can write \eqref{lemma bound} for $x^*=x_\#$. Invoking Lemma \eqref{a.s.} again we conclude that $\|x_k-x_\#\|^2$ is a convergent sequence, i.e., there exists $a\geq 0$ such that $a=\lim_{k\to \infty}\|x_k-x_\#\|^2$. Since ${x_{k_q}}$ is a subsequance of ${x_k}$, then we know that $a=\lim_{k\to \infty}\|x_k-x_\#\|^2=\lim_{q\to \infty} \|x_{k_q}-x_\#\|^2=0$. 
    %Moreover, we knew that $x_{k_q}\to x_\#$, so $\lim_{q\to \infty} \|x_{k_q}-x_\#\|^2=0$ which implies that $a=0$. 
    Therefore, we conclude that $\lim_{k\to \infty} \|x_k-x_\#\|^2=0$ almost surely which means that $x_k\to x_\#$ almost surely.
\end{proof}
Now, in Theorem \ref{th:rate} we demonstrate the convergence rate of the proposed method. 
\begin{theorem}[Convergence rate]\label{th:rate}
Let $\{x_k\}_{k\geq 0}$ be the iterates generated by Algorithm \ref{alg1} and suppose Assumptions \ref{asump1} and \ref{assump_error} hold. Choose $\eta_k = \eta \geq \tfrac{1}{2\mu}$.
Let $N_k= \lceil (k+1)^{2+2\delta}\rceil$ where $\delta > 0 $ for all $k \geq 0$, Then the following holds:
\begin{align*}
 \min_{k\in\{0,\hdots,T-1\}}\mathbb E[\|H(x_k,\eta)\|^2] \leq \tfrac{1}{ T\left(1-\tfrac{1}{2\eta \mu}\right)}[\|x_0-x^*\|^2+\tfrac{\pi^2\nu^2}{\eta^2}+\tfrac{2\nu \|x^*\|}{\eta } (1+\z{\tfrac{1}{\delta}})] = \mathcal O(1/T).
\end{align*}
\end{theorem}
\begin{proof}
Taking expectations from both sides of \eqref{lemma bound} and use the fact that $\mathbb E[\mathbb E[\|x_{k+1}-x^*\|^2\mid \mathcal F_k]] = \mathbb E[\|x_{k+1}-x^*\|^2]$, we have: 
\begin{align*}
    \mathbb E[\|x_{k+1}-x^*\|^2]\leq \mathbb E[\|x_k-x^*\|^2]-\left(1-\tfrac{1}{2\eta_k\mu}\right)\mathbb E[\|H(x_k,\eta_k)\|^2]+\tfrac{\z{6}\nu^2}{\eta_k^2N_k}+\tfrac{2\nu}{\eta_k\sqrt{N_K}}\|x^*\|,
\end{align*}
Now, by summing over $k= 0, \dots,T-1$ and putting $\eta_k = \eta$ the following holds: 
\begin{align*}
\sum_{k=0}^{T-1}\left(1-\tfrac{1}{2\eta \mu}\right)\mathbb E[\|H(x_k,\eta)\|^2] \leq \|x_0-x^*\|^2+\sum_{k=0}^{T-1}\tfrac{\z{6}\nu^2}{\eta^2N_k}+\sum_{k=0}^{T-1}\tfrac{2\nu}{\eta \sqrt{N_K}}\|x^*\|,
\end{align*}
Now, we can provide a lower-bound for the left-hand-side of the above inequality by $\min_k \mathbb E[\|H(x_k,\eta)\|^2]$, then we have:
\begin{align}\label{bound before deviding}
T\left(1-\tfrac{1}{2\eta \mu}\right) \min_k\mathbb E[\|H(x_k,\eta)\|^2] \leq \|x_0-x^*\|^2+\sum_{k=0}^{T-1}\tfrac{\z{6}\nu^2}{\eta^2N_k}+\sum_{k=0}^{T-1}\tfrac{2\nu}{\eta \sqrt{N_K}}\|x^*\|.
\end{align}
Let $\eta > \tfrac{1}{\z{2}\mu}$, and select $N_k = \lceil (k+1)^{2+2\delta}\rceil$, for some $\delta >0 $. Therefore, one can deduce the following inequalities using simple algebra.
\begin{align*}
&\sum_{k=0}^{T-1}\tfrac{1}{N_k} \leq\sum_{k=0}^{T-1}\tfrac{1}{(k+1)^{2+2\delta}} \leq \tfrac{\pi^2}{6},\\
&\sum_{k=0}^{T-1}\tfrac{1}{\sqrt{N_k}} \leq\sum_{k=0}^{T-1}\tfrac{1}{(k+1)^{1+\delta}} \leq 1+ \sum_{k=1}^{T-1}\tfrac{1}{(k+1)^{1+\delta}} \leq 1+\int_{0}^{T-1} \tfrac{1}{(x+1)^{1+\delta}} \,dx \leq 1+\tfrac{1}{\delta}-\tfrac{1}{T\delta} \leq 1+\tfrac{1}{\delta}.    
\end{align*}
Next, using the above inequalities in \eqref{bound before deviding}, we  obtain 
\begin{align*}
T\left(1-\tfrac{1}{2\eta \mu}\right) \min_k\mathbb E[\|H(x_k,\eta)\|^2] \leq \|x_0-x^*\|^2+\tfrac{\pi^2\nu^2}{\eta^2}+\tfrac{2\nu \|x^*\|}{\eta } (1+\z{\tfrac{1}{\delta}}).
\end{align*}
Finally, dividing both sides of previous inequality by $T\left(1-\tfrac{1}{2\eta \mu}\right)$ leads to 
\begin{align*}
    \min_{k\in\{0,\hdots,T-1\}}\mathbb E[\|H(x_k,\eta)\|^2] \leq \tfrac{1}{ T\left(1-\tfrac{1}{2\eta \mu}\right)}[\|x_0-x^*\|^2+\tfrac{\pi^2\nu^2}{\eta^2}+\tfrac{2\nu \|x^*\|}{\eta } (1+\z{\tfrac{1}{\delta}})] = \mathcal O(1/T).
\end{align*}
\end{proof}

\begin{remark}
In Theorem \ref{th:rate}, we established the convergence rate of the proposed algorithm. In particular, we showed that the proposed method finds an $\epsilon$-gap in $\mathcal O(1/\epsilon)$ iterations. Since at each iteration of the method, $N_k$ samples are required, the total number of samples (oracle complexity) to achieve $\epsilon$-gap is $\sum_{k=0}^{T-1} N_k\geq \int_{0}^{T-1} (x+1)^{2+2\delta} dx=\mathcal O(1/\epsilon^{3+2\delta})$. 
\end{remark}

\section{NUMERICAL EXPERIMENTS}\label{sec:numeric}
In this section, to illustrate the effectiveness of the proposed scheme, we first solve a test problem (Example 1) and then we solve a network equilibrium control problem (Example 2). The simulations are conducted in MATLAB. 

{\bf Example 1.} Consider problem \eqref{sIvI} for $G(x,\xi)=Ax+b(\xi)$, where
$$A = \begin{pmatrix} 5.0 & 2.0 & 1.0 \\ 2.0 & 5.0 & 0.0 \\ 1.0 & 0.0 & 6.0 \end{pmatrix},\quad  b(\xi) = \begin{pmatrix} 0.0\\ -3.0 \\ -5.5\end{pmatrix}+\xi,$$  $\xi\in\mathbb R^3$ is generated randomly from a standard Gaussian distribution, and $X=\{x\in \mathbb R^3\mid x_i\in [-1,10],\ i=1,2,3\}$. It is easy to check that $F(x)=\mathbb E[G(x,\xi)]$ is a co-coercive map and the unique optimal solution of the problem is $x^*=[0,0.4,0.75]^T$. Our proposed method is implemented for 20 replications with the same starting point and the mean of replications are plotted.  Figure \ref{fig1} (left) shows  the convergence of VR-IPG  for different choices of stepsizes when sample size $N_k=(k+1)^{2+2\delta}$ with $\delta=1/2$. 

\begin{figure}[htb]
    \centering
    \includegraphics[scale=0.113]{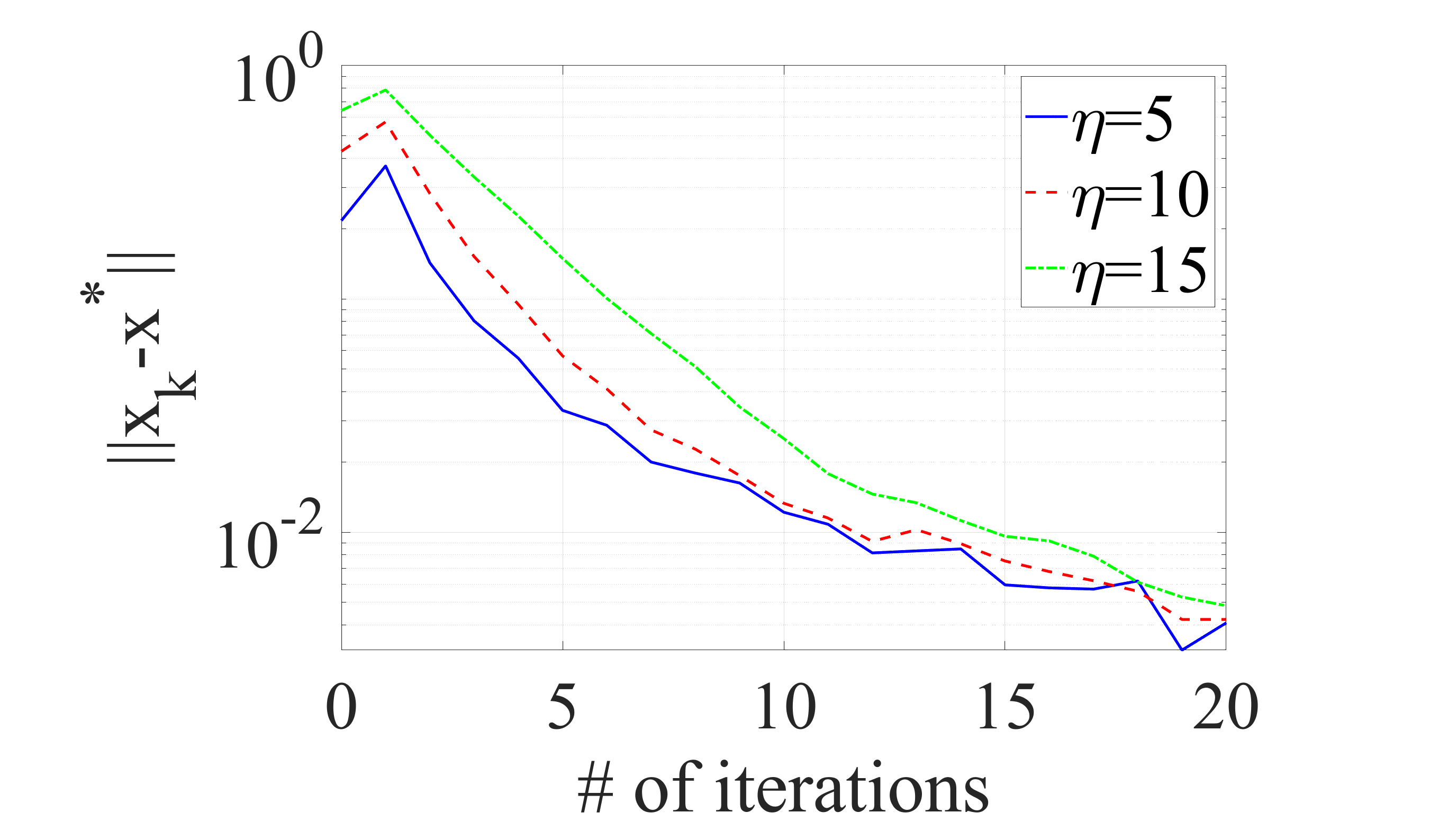}
        \includegraphics[scale=0.113]{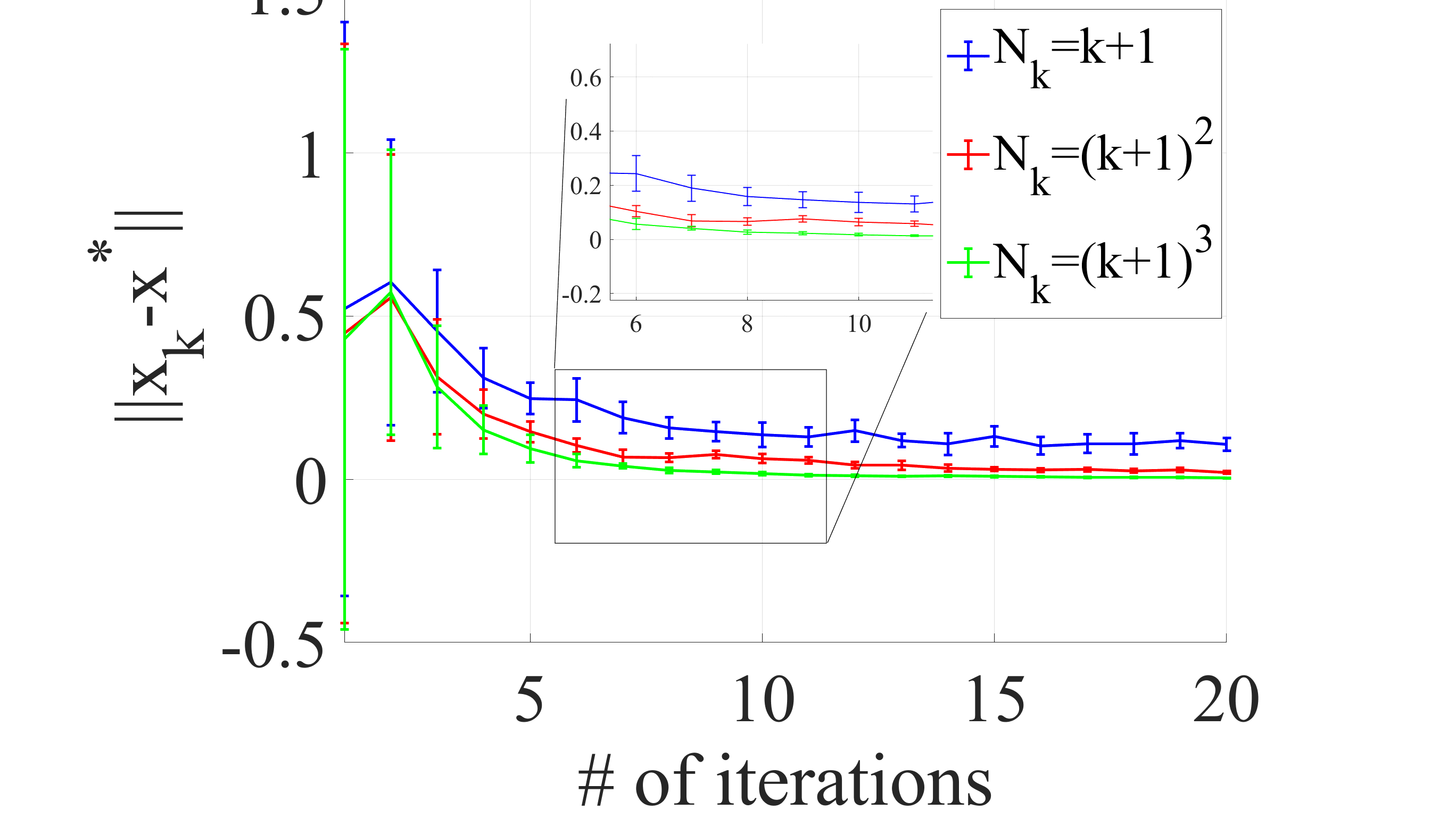}
    \caption{Example 1: Convergence of VR-IPG for different choices of stepsize (left); Almost sure convergence of VR-IPG for different sample sizes. }
    \label{fig1}
\end{figure}
We also implement VR-IPG for different choices of $N_k$ and the mean of 20 replications and their 95\% confidence intervals are plotted in Figure \ref{fig1} (right). From the plots, we observe that the confidence bands are becoming tighter as $k\to \infty$.

%\vspace{\baselineskip}
%MATLAB was used to find the optimal solution this problem. 
%We first defined u to be a line vector of 1's.
%u = $\begin{pmatrix} 1 \\ 1 \\ 1 \end{pmatrix}$;

%We had three constants: numiter = 100 for the number of iterations of the first for loop. B = 10 which was a constant for solving the inverse inequality. Finally, val = 0, this constant was used in the second for loop to check the value of matrix C.

%Then, we found the initial F and r and they will be updated with every iteration of the for loop. We have an equation C which was used to simplify the equation for r which is also updated with every iteration of the for loop. In the first for loop u, F and R are updated with every iteration. Then, the second for loop checks the contents of matrix C to ensure that it falls within the range of u for calculations.

%In the end, we found the solution to be $u^*$ = $[0, 0.4, 0.75]^T$.

{\bf Example 2.} In this example, we consider a network equilibrium control problem  described in \cite{he2010solving}. The problem is motivated by calculating the equilibrium in a supply and demand network. The goal is to control resource exploitations on the supply market and commodity consumptions on the demand markets by regulating taxes as a control variable denoted by $x_1\in\mathbb R^m$ and $x_2\in\mathbb R^n$, respectively. Let $x=[x_1^\top~ x_2^\top]^\top\in\mathbb R^{m+n}$ and the commodity shipment between supply and demand markets are denoted by ${\bf a}=[[a_{ij}]_{i=1}^n]_{j=1}^m\in\mathbb R^{nm}$. The problem can be formulated as an SIVI in \eqref{sIvI} where the constraint set is described by $X = \{ x \in \mathbb R^{m+n} \mid Lx \leq b,~ F^{\min} \leq x \leq F^{\max} \}$ which restricts the total resource consumption for some $L\in \mathbb R^{q\times (m+n)}$,  $b\in \mathbb R^q$, and $q>0$, moreover, the map $F$ is described by a linear model based on an optimal unit of shipments ${\bf a}^*$. Different from \cite{he2010solving}, here we consider a stochastic linear model to include uncertainty. In particular, let $G(x,\xi)=[(A{\bf a}^*(x)+\xi_1)^\top,~(B{\bf a}^*(x)+\xi_2)^\top]^\top$ where $\xi_1\in \mathbb R^m$, $\xi_2\in \mathbb R^n$ are random variables with standard Gaussian distribution, and 
\begin{equation*}
A \triangleq 
\begin{bmatrix}
e_n^T & 0 & \cdots & 0 \\
0 & e_n^T & \cdots & 0 \\
\vdots & \vdots & \ddots & \vdots \\
0 & 0 & \cdots & e_n^T
\end{bmatrix} \in \mathbb R^{m*mn}
, \qquad B \triangleq [I_n,\dots,I_n] \in \mathbb R^{n*mn}.
\end{equation*} 
For a given tax policy $x$, ${\bf a}^*(x)$ can be obtained by solving the following equilibrium problem 
\begin{align*}
   \langle {\bf a}'-{\bf a}, \ell({\bf a})+A^T(g({\bf a})+(\alpha+x_1))-B^T(h({\bf a})-(\beta+x_2)) \rangle \geq 0,\quad \forall {\bf a}' \geq 0.
\end{align*}
where $\alpha\in\mathbb R^m$ is the resource tax imposed on the supply market, $\beta\in\mathbb R^n$ is the sale tax imposed in the demand market, $\ell(x)\triangleq \mbox{diag}(c)x+\tau$, $g({\bf a})\triangleq \mbox{diag}(a)A{\bf a}+a^0$, and $h({\bf a})\triangleq \mbox{diag}(\rho)B{\bf a}+\rho^0$ for some $c\in\mathbb R^{nm}$, $a,a^0\in\mathbb R^m$ and $\rho,\rho^0\in \mathbb R^n$. In this experiment, similar to \cite{he2010solving} we set $m=10$, $n=30$, $q=2$, $F^{\min}=[\mathbf{0}_m^\top~20e_n^\top]^\top$, $F^{\max}=[160e_m^\top~60e_n^\top]^\top$, and other problem's parameter are generated uniformly at random according to Table \ref{tab:param}.

\begin{table}[htb]
\centering
\caption{Distribution of randomly generated problem's parameters.}
\begin{tabular}{|c|c|c|c|c|c|c|}
\hline
Parameter & $c$ & $\tau$ & $a$ & $a^0$ & $\rho$ & $\rho^0$ \\ \hline
Distribution & $U[0.1,0.2]$ & $U[1,2]$ & $U[1,2]$ & $U[270,370]$ & $U[1,2]$ & $U[620,720]$ \\ \hline
\end{tabular}
\label{tab:param}
\end{table}

\z{We execute the code 20 times with the same initial point $x_0 = \mathbf{0}$, and visualize the mean of replications. Figure \ref{fig2} (left) depicts the convergence of VR-IPG for varying step sizes when the sample size $N_k = (k+1)^{2+2\delta}$ and $\delta=1/2$.}
\begin{figure}[htb]
    \centering
    \includegraphics[scale=0.112]{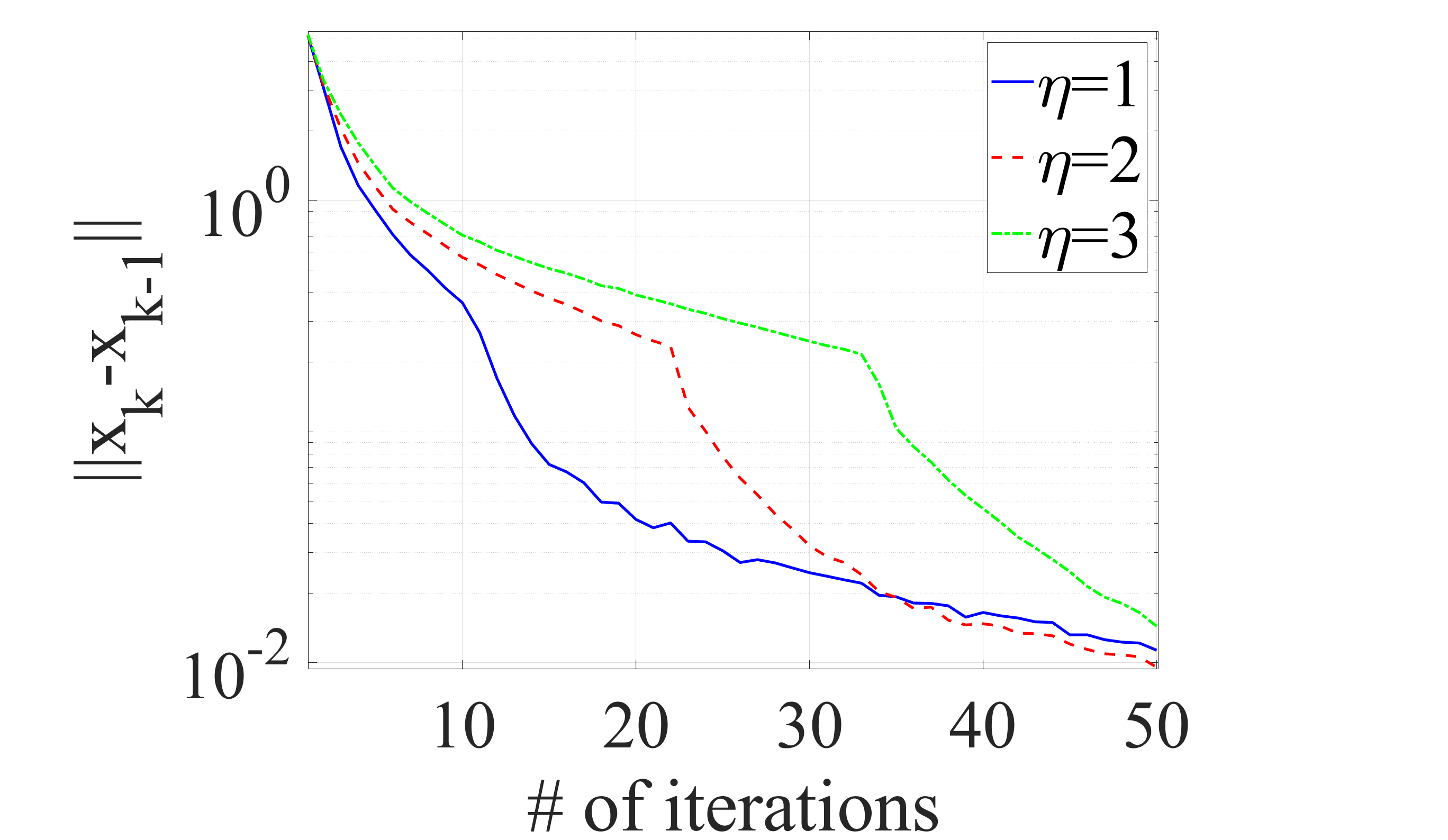}
        \includegraphics[scale=0.114]{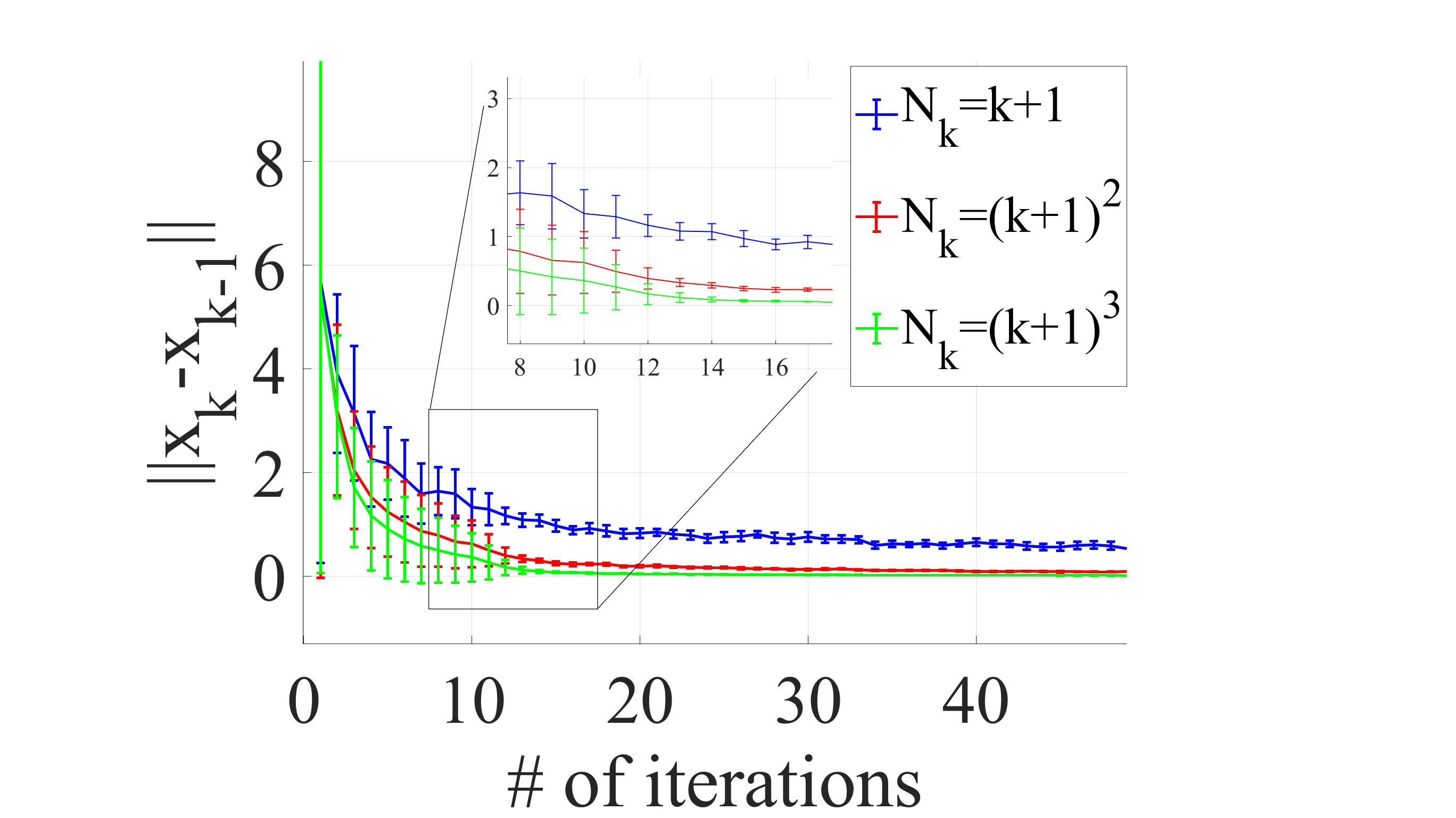}
    \caption{Convergence of VR-IPG implemented on Example 2 for different choices of stepsize (left); Almost sure convergence of VR-IPG for different sample sizes (right). }
    \label{fig2}
\end{figure}
\z{In addition, Figure \ref{fig2} (right) depicts the results of implementing VR-IPG for a variety of $N_k$ values, along with the mean of 20 replications, with stepsize equal to 1 and their respective 95\% confidence intervals. We observe that as $k$ goes to infinity, the confidence intervals become tighter.}

%\section{CONCLUDING REMARKS}\label{sec:conclude}
%In this paper, we focus on problems involving co-coercive stochastic inverse-variational inequality (SIVI)...

\appendix

\section{APPENDICES}\label{appendix}
In this section, proof of Lemma \ref{first bound} and Lemma \ref{lem:bound} are provided.

{\bf Proof of Lemma \ref{first bound}}.
Since $x^*$ is a solution and $\mathbf{P}_X(F(x)-\eta x) \in X $ by using definition of IVI we have: 
\begin{align}\label{def IVI}
    \langle \eta x^*,\mathbf{P}_X(F(x)-\eta x) - F(x^*)\rangle  \geq 0.
\end{align}
Choosing $u = F(x)-\eta x $ in Lemma \ref{proj}. Since $F(x^*) \in X$ then following holds:
\begin{align} \label{using_lemma1}
    \langle (F(x)-\eta x) - \mathbf{P}_{X}( F(x)-\eta x), \mathbf{P}_{X}( F(x)-\eta x)- F(x^*)\rangle \geq 0.
\end{align}
Combining \eqref{def IVI} and \eqref{using_lemma1}, we get : 
\begin{align*}
\langle \eta H(x,\eta)- \eta (x-x^*), (F(x)-F(x^*))- \eta H(x, \eta)\rangle \geq 0. 
\end{align*}
and consequently
\begin{align*}
(x-x^*)^T H(x,\eta) \geq \| H(x,\eta)\|^2+\tfrac{1}{\eta}(x-x^*)^T(F(x)-F(x^*))-\tfrac{1}{\eta}(F(x)-F(x^*))^T H(x,\eta).
\end{align*}
Applying the co-coercivity of mapping $F$, we obtain:
\begin{align*}
(x-x^*)^T H(x,\eta) &\geq  \| H(x,\eta)\|^2+ \tfrac{\mu}{\eta}\|F(x)-F(x^*)\|^2-\tfrac{1}{\eta}(F(x)-F(x^*))^T H(x,\eta)\\
&=  \| H(x,\eta)\|^2- \tfrac{1}{4 \eta \mu} \| H(x,\eta)\|^2 +\|\sqrt{\tfrac{\mu}{\eta}} (F(x)-F(x^*))- \tfrac{1}{2\sqrt{\mu\eta}} H(x,\eta)\|^2\\
&\geq (1-\tfrac{1}{4 \eta \mu}) \| H(x,\eta)\|^2.\qedd
\end{align*}

{\bf Proof of Lemma \ref{lem:bound}.} 
Recall that $\bar w_{k,N_k}= \tfrac{1}{N_k}{\sum_{j=1}^{N_k} ( G(x_k,\xi_{j,k})-F(x_k) )}$. Using the update rule of $x_{k+1}$ in Algorithm \ref{alg1}, one can obtain the following.
\begin{align*}
    &\|x_{k+1}-x^*\|^2\\&= \|x_k-x^*-\tfrac{1}{\eta_k}\left(F(x_k)+\bar w_{k,N_k}-\mathbf{P}_X\left(F(x_k)+\bar w_{k,N_k}-\eta_kx_k\right)\right)\|^2\\&
    =\|x_k-x^*-\tfrac{1}{\eta_k}\left(F(x_k)-\mathbf{P}_X\left(F(x_k)-\eta_kx_k\right)\right)-\tfrac{1}{\eta_k}\big(\bar w_{k,N_k}-\mathbf{P}_X\left(F(x_k)+\bar w_{k,N_k}-\eta_kx_k\right)\\
    &\quad\mathbf +P_X\left(F(x_k)-\eta_kx_k\right)\big)\|^2. 
    \end{align*}
For notational simplicity, define $e_k\triangleq \bar w_{k,N_k}\z{-}\mathbf{P}_X\left(F(x_k)+\bar w_{k,N_k}-\eta_kx_k\right)\z{+}\mathbf P_X\left(F(x_k)-\eta_kx_k\right)$ and  $H(x_k,\eta_k)\triangleq \tfrac{1}{\eta_k}\left(F(x_k)-\mathbf{P}_X\left(F(x_k)-\eta_kx_k\right)\right)$. Then, using the fact that $(a-b-c)^2=a^2+b^2-2ab+c^2-2ac+2bc$, one can obtain
    \begin{align}\label{bound 1}
\nonumber    &\|x_{k+1}-x^*\|^2\\&
    \leq \|x_k-x^*\|^2+\|H(x_k,\eta_k)\|^2-2\langle x_k-x^*,H(x_k,\eta_k)\rangle+\frac{1}{\eta_k^2}\|e_k\|^2-\tfrac{2}{\eta_k}\langle x_k-x^*,e_k\rangle+\tfrac{2}{\eta_k}\langle H(x_k,\eta_k),e_k\rangle.
    \end{align}
Using the definition of $e_k$ and Lemma \ref{proj}, we have that 
\begin{align}\label{bound e}
    \frac{1}{\eta_k^2}\|e_k\|^2\leq \tfrac{2}{\eta_k^2}\|\bar w_{k,N_k}\|^2+\tfrac{2}{\eta_k^2}\|\mathbf{P}_X\left(F(x_k)+\bar w_{k,N_k}-\eta_kx_k\right)-\mathbf P_X\left(F(x_k)-\eta_kx_k\right)\|^2\leq \tfrac{4}{\eta_k^2}\|\bar w_{k,N_k}\|^2.
\end{align}
Moreover, from definition of $e_k$ and $H(x_k,\eta_k)$, one can show that
\begin{align}\label{bound 2}
\nonumber&\tfrac{2}{\eta_k}\langle H(x_k,\eta_k)-(x_k-x^*),\z{e_k}\rangle\\ \nonumber&=\tfrac{2}{\eta_k}\langle H(x_k,\eta_k)-(x_k-x^*),\bar w_{k,N_k}\rangle\\&\nonumber\quad\z{-}\langle \tfrac{2}{\eta_k}(H(x_k,\eta_k)-(x_k-x^*)),\mathbf{P}_X\left(F(x_k)+\bar w_{k,N_k}-\eta_kx_k\right)-\mathbf P_X\left(F(x_k)-\eta_kx_k\right)\rangle \\&\nonumber
=\tfrac{2}{\eta_k}\langle H(x_k,\eta_k)-(x_k-x^*),\bar w_{k,N_k}\rangle\\ \nonumber&\quad -\langle \tfrac{2}{\eta_k^2} \z{(}(F(x_k)-\eta_kx_k)+\eta_kx^*-\mathbf P_X(F(x_k)-\eta_kx_k) \z{)},\mathbf{P}_X\left(F(x_k)+\bar w_{k,N_k}-\eta_kx_k\right)-\mathbf P_X\left(F(x_k)-\eta_kx_k\right)\rangle, \end{align}
By adding and subtracting $\bar w_{k,N_k}$ and $\mathbf{P}_X\left(F(x_k)+\bar w_{k,N_k}-\eta_kx_k\right)$, we obtain:
\begin{align}
\nonumber&\tfrac{2}{\eta_k}\langle H(x_k,\eta_k)-(x_k-x^*),\z{e_k}\rangle\\ 
\nonumber&
=\tfrac{2}{\eta_k}\langle H(x_k,\eta_k)-(x_k-x^*),\bar w_{k,N_k}\rangle\\
\nonumber&\quad \underbrace{- \tfrac{2}{\eta_k^2} \langle  (F(x_k)\z{+\bar w_{k,N_k}}-\eta_kx_k)-\z{\mathbf{P}_X\left(F(x_k)+\bar w_{k,N_k}-\eta_kx_k\right)}, \mathbf{P}_X\left(F(x_k)+\bar w_{k,N_k}-\eta_kx_k\right)-\mathbf P_X\left(F(x_k)-\eta_kx_k\right)\rangle}_{\text{term (a)}}\\
\nonumber &\quad  -\tfrac{2}{\eta_k^2} \langle  \z{\mathbf{P}_X\left(F(x_k)+\bar w_{k,N_k}-\eta_kx_k\right)}-\mathbf P_X(F(x_k)-\eta_kx_k) ,\mathbf{P}_X\left(F(x_k)+\bar w_{k,N_k}-\eta_kx_k\right)-\mathbf P_X\left(F(x_k)-\eta_kx_k\right)\rangle\\
\nonumber &\quad  +\tfrac{2}{\eta_k^2} \langle  \z{\bar w_{k,N_k}},\mathbf{P}_X\left(F(x_k)+\bar w_{k,N_k}-\eta_kx_k\right)-\mathbf P_X\left(F(x_k)-\eta_kx_k\right)\rangle\\
\nonumber &\quad\z{-}\tfrac{2}{\eta_k}\langle x^*,\mathbf{P}_X\left(F(x_k)+\bar w_{k,N_k}-\eta_kx_k\right)-\mathbf P_X\left(F(x_k)-\eta_kx_k )\right\rangle,
\end{align}
Based on Lemma \ref{proj} part (b), one can easily verify that $\mbox{term (a)}\leq0$,  
then the following can be obtained:
\begin{align}
\nonumber&\tfrac{2}{\eta_k}\langle H(x_k,\eta_k)-(x_k-x^*),\z{e_k}\rangle\\
\nonumber&
\leq \tfrac{2}{\eta_k}\langle H(x_k,\eta_k)-(x_k-x^*),\bar w_{k,N_k}\rangle \z{-}\tfrac{2}{\eta_k}\langle x^*,\mathbf{P}_X\left(F(x_k)+\bar w_{k,N_k}-\eta_kx_k\right)-\mathbf P_X\left(F(x_k)-\eta_kx_k \z{)}\right\rangle\\
\nonumber &\quad - \tfrac{2}{\eta_k^2} \langle  \z{\mathbf{P}_X\left(F(x_k)+\bar w_{k,N_k}-\eta_kx_k\right)}-\mathbf P_X(F(x_k)-\eta_kx_k),
\mathbf{P}_X\left(F(x_k)+\bar w_{k,N_k}-\eta_kx_k\right)-\mathbf P_X\left(F(x_k)-\eta_kx_k\right)\rangle\\
\nonumber & \quad + \tfrac{2}{\eta_k^2} \langle  \bar w_{k,N_k},
\mathbf{P}_X\left(F(x_k)+\bar w_{k,N_k}-\eta_kx_k\right)-\mathbf P_X\left(F(x_k)-\eta_kx_k\right)\rangle\\
\nonumber&
= \tfrac{2}{\eta_k}\langle H(x_k,\eta_k)-(x_k-x^*),\bar w_{k,N_k}\rangle \z{-}\tfrac{2}{\eta_k}\langle x^*,\mathbf{P}_X\left(F(x_k)+\bar w_{k,N_k}-\eta_kx_k\right)-\mathbf P_X\left(F(x_k)-\eta_kx_k \z{)}\right\rangle\\
\nonumber &\quad  \underbrace{- \tfrac{2}{\eta_k^2} \|  \mathbf{P}_X\left(F(x_k)+\bar w_{k,N_k}-\eta_kx_k\right)-\mathbf P_X(F(x_k)-\eta_kx_k)\|^2 }_{\text{term (b)}}\\&\nonumber\quad+ \tfrac{2}{\eta_k^2} \langle  \bar w_{k,N_k},
\mathbf{P}_X\left(F(x_k)+\bar w_{k,N_k}-\eta_kx_k\right)-\mathbf P_X\left(F(x_k)-\eta_kx_k\right)\rangle\\
&\leq \z{\tfrac{2}{\eta_k}\langle H(x_k,\eta_k)-(x_k-x^*),\bar w_{k,N_k}\rangle +}\tfrac{2}{\eta_k}\|x^*\|\|\bar w_{k,N_k}\|\z{+\tfrac{2}{\eta_k ^2}\|\bar w_{k,N_k}\|^2 },
\end{align}
where in the last two inequalities we used the fact that $\text{term (b)} \leq 0$, Lemma \ref{proj} and Cauchy-Schwartz inequality. Using Lemma \ref{first bound}, inequalities \eqref{bound e} and \eqref{bound 2} in \eqref{bound 1}, we obtain
\begin{align*}
\|x_{k+1}-x^*\|^2&\leq \|x_k-x^*\|^2-\left(1-\tfrac{1}{2\eta_k\mu}\right)\|H(x_k,\eta_k)\|^2+\z{\tfrac{2}{\eta_k}\langle H(x_k,\eta_k)-(x_k-x^*),\bar w_{k,N_k}\rangle }\\&\nonumber\quad+\tfrac{\z{6}}{\eta_k^2}\|\bar w_{k,N_k}\|^2+\tfrac{2}{\eta_k}\|x^*\|\|\bar w_{k,N_k}\|.
\end{align*}
 Taking conditional expectation and using Assumption \ref{assump_error}, we have the following.
\begin{align*}
    \mathbb E[\|x_{k+1}-x^*\|^2\mid \mathcal F_k]\leq \|x_k-x^*\|^2-\left(1-\tfrac{1}{2\eta_k\mu}\right)\|H(x_k,\eta_k)\|^2+\tfrac{\z{6}\nu^2}{\eta_k^2N_k}+\tfrac{2\nu}{\eta_k\sqrt{N_K}}\|x^*\|.
    \qedd
\end{align*}

\bibliographystyle{siam}
\bibliography{demobib}

\end{document}